\documentclass[12pt, a4paper]{amsart}

\usepackage{amsmath}
\usepackage{amssymb}
\usepackage{amsthm}
\usepackage{amscd}
\usepackage{amsfonts}
\usepackage{amsbsy}
\usepackage{epsfig}
\usepackage{enumerate}
\usepackage{soul}
\usepackage[normalem]{ulem}
\usepackage[shortlabels]{enumitem}
\usepackage[latin1]{inputenc}
\usepackage{tikz}
\usepackage{hyperref}

\newtheorem{theorem}{Theorem}[section]
\newtheorem{lemma}[theorem]{Lemma}

\newtheorem{proposition}[theorem]{Proposition}
\newtheorem{corollary}[theorem]{Corollary}

\newtheorem{remark}[theorem]{Remark}

\newtheorem{definition}[theorem]{Definition}
\newtheorem{question}[theorem]{Question}
\newtheorem*{claim*}{Claim}

\newfont\bbf{msbm10 at 12pt}

\def\eps{\varepsilon}

\def\AC{{\hbox{{\rm AC}}}}
\def\VB{{\hbox{{\rm VB}}}}
\def\ACG{{\hbox{{\rm ACG}}}}
\def\VBG{{\hbox{{\rm VBG}}}}
\def\ap{\mbox{\tiny  ap}}

\def\id{{\hbox{{\rm id}}}}

\def\supp{\mbox{\rm supp}}

\def\inter{\mbox{\rm int} }

\DeclareMathOperator{\conv}{conv}

\def\le{\leqslant}
\def\ge{\geqslant}

\def\per{\mathrm{Per}}

\def\1{ {\hbox{{\it 1}} \!\! I} }

\def\CP{\overline{CP}}

\begin{document}

\title[On Besicovitch functions]
{On Lebesgue measure preserving Besicovitch functions}
\author{Jozef Bobok}
\author{Jernej \v Cin\v c}
\author{Piotr Oprocha}
\author{Serge Troubetzkoy}

\date{\today}

\address[J.\ Bobok]{Department of Mathematics of FCE, Czech Technical University in Prague,
Th\'akurova 7, 166 29 Prague 6, Czech Republic}
\email{jozef.bobok@cvut.cz}

\address[J.\ \v{C}in\v{c}]{University of Maribor, Koro\v ska 160, 2000 Maribor, Slovenia
 -- $\&$ -- 
	Abdus Salam International Centre for Theoretical Physics (ICTP), Trieste, Italy}
\email{jernej.cinc@um.si}

\address[P.\ Oprocha]{
Centre of Excellence IT4Innovations - Institute for Research and Applications of Fuzzy Modeling, University of Ostrava, 30. dubna 22, 701 03 Ostrava 1, Czech Republic}
\email{piotr.oprocha@osu.cz}

\address[S.\ Troubetzkoy]{Aix Marseille Univ, CNRS, I2M, Marseille, France}
\email{serge.troubetzkoy@univ-amu.fr}
	
\begin{abstract} We consider the space $C_{\lambda}$ of all continuous interval maps preserving the Lebesgue measure $\lambda$. A continuous function $f\colon~[0,1]\to \mathbb R$ is called Besicovitch if it does not have any finite or infinite unilateral derivative. It is known that the set of Besicovitch functions in $C_{\lambda}$ is nonempty and meager.  

We prove that no Besicovitch function is invertible $\lambda$-almost everywhere.  As a consequence, every Besicovitch function in $C_{\lambda}$ has positive measure-theoretic entropy with respect to $\lambda$.
 Furthermore, we show that Besicovitch functions are dense in $C_{\lambda}$ and, consequently, also dense in the class of interval maps with a dense set of periodic points.
\end{abstract}

\maketitle
\section{introduction}
We consider $C_\lambda$, the set of continuous maps of the interval $I := [0,1]$ which preserve the Lebesgue measure $\lambda$.
Except for the identity map and the reflection $1 - \id$ every map in $C_\lambda$ 
has topological entropy at least $\log2/2$ and generically they have infinite 
topological entropy \cite{BT}.
In the recent article \cite{BCOT} the authors have shown that the
generic map in $C_{\lambda}$ has zero measure-theoretic
entropy with respect to $\lambda$.
This result is quite striking, since the graphs of generic maps in 
$C_\lambda$ are highly irregular:
they have upper box dimension 2, iterated $\delta$-crookedness property and have knot points $\lambda$-a.e.\ \cite{BCOT1}. 
It is therefore natural to ask whether there exists any relation between the 
irregular structure of the graph of a map and its measure-theoretic entropy with 
respect to $\lambda$.
The results of \cite{BCOT} suggest that such a connection, if it exists, is far 
from straightforward, and is generically even  counterintuitive.
This question forms both the starting point and the motivation of the present investigation.

One can attempt to address this question for different notions of irregular behavior of the graph of a map. 
In this article, we study it in the context of Besicovitch functions.
Our main result, Theorem \ref{t:3}, states that no Besicovitch function is invertible 
$\lambda$-a.e.; as a consequence, every Besicovitch function in 
$C_\lambda$ has positive measure-theoretic entropy with respect to $\lambda$ (Corollary \ref{c1}).
Besicovitch functions are continuous functions that are nowhere one-sided differentiable, neither finitely nor infinitely.  For an extensive exposition of Besicovitch functions we recommend the  book \cite[Ch. 11]{JP16}
(see also \cite{C}, \cite{H}).  A direct construction of a Besicovitch function 
in $C_{\lambda}$ was given by Bobok \cite{Bo90}, \cite{Bo04}, hence the set of such maps
in $C_\lambda$ is nonempty. As a complement to our main result
 we prove that the set of Besicovitch functions is dense in 
$C_{\lambda}$ (Proposition \ref{p:2}); as a consequence we obtain that the set of Besicovich functions is dense also in the set of all continuous interval maps having dense set of periodic points.
 Saks has shown that the collection of all Besicovitch functions is a meager set in $C(I)$ \cite{Sak32}; the same result also holds in $C_{\lambda}$ \cite[Section 3]{Bo90}. 
 In Corollary \ref{c2} we  give a new proof of this fact. We conclude the article with a list
 of related open questions.

\section{Basic definitions and known results}

We say that a point $x\in (0,1)$ is a {\em point of density of a measurable set $A\subset I$} if
$$\lim_{h\to 0_{+}}\frac{\lambda(A\cap [x-h,x+h])}{2h}=1$$
while we say that $0$, resp.\ $1$, is a point of density of $A$ if 
$$\lim_{h\to 0_{+}}\frac{\lambda(A\cap [0,h])}{h}=1, \quad \text{resp.} \quad \lim_{h\to 0_{+}}\frac{\lambda(A\cap [1-h,1])}{h}=1.$$
The set of all points of density of a measurable set $A\subset I$ will be denoted by $D(A)$.

\begin{lemma}\label{l:3}If $A\subset I$ is measurable then $D(A)$ is Borel.
\end{lemma}
\begin{proof} Clearly
\begin{equation*}
D(A)=\bigcap_{n\ge 1}\bigcup_{\ell\ge 1}\bigcap_{m\ge\ell}A_{m,n},
\end{equation*}
where each $A_{m,n}=\{x\in I\colon~\frac{m}{2}\lambda(A\cap [x-\frac{1}{m},x+\frac{1}{m}])>1-\frac{1}{n}\}$ is open.
\end{proof}

\begin{remark}\label{rem:fullD}
For each 
measurable set $A\subset I$ one has $\lambda(A)=\lambda(D(A))$  \cite[Theorem 5.1]{AMB}.
\end{remark}

\begin{definition}Let $f\colon~I\to\mathbb R$ be measurable and $x_0\in (0,1)$. If there is a measurable set $A\subset I$ for which $x_0$ is a point of density and the limit
$$\lim_{x\to x_0,~ x\in A}\frac{f(x)-f(x_0)}{x-x_0}$$
exists, we will call it the {\em approximate derivative of $f$ at $x_0$} and denote it $f'_{ap}(x_0)$. We say $f$ is {\em approximately differentiable at $x_0$} if $f'_{ap}(x_0)\in\mathbb R$.
 \end{definition}

 Recall that \emph{Dini derivatives} are defined by $D^\pm f(x)=\limsup_{x\to (x_0)_{\pm}} \frac{f(x)-f(x_0)}{x-x_0}$ and 
$D_{\pm} f(x)=\liminf_{x\to (x_0)_{\pm}} \frac{f(x)-f(x_0)}{x-x_0}$. 
A \emph{knot point} of function $f\in C(I)$ is every $x\in I$ where Dini derivatives satisfy $D^+f(x) = D^-f(x) = \infty$ and $D_+f(x) = D_-f(x) = -\infty$. Note that knot point at $x$ is a special case of the situation when $f'(x)$ does not exist.

\begin{definition}Let $f\colon~\mathbb R\to \mathbb R$ and let $E\subset \mathbb{R}$. Then we write
\begin{equation*}V (f;E) = \sup\sum_{i=1}^n\vert f(d_i)-f(c_i)\vert,
\end{equation*}
where the supremum is computed using all choices of 
 finite collections of intervals with pairwise disjoint interiors
from the collection of all intervals $[c,d]$ with endpoints $c$ and $d$ in the set $E$.
\end{definition}

The following concepts and abbreviations follow Saks' terminology \cite[Ch. 7]{Sak37}.

\begin{definition}Let $f\colon~\mathbb R\to\mathbb R$ and let $E$ be a set of real numbers. We say that f  has $\VB$ ({\em bounded variation}) on $E$ provided that $V(f;E)<\infty$. We say that $f$  has $\VBG$ ({\em generalized bounded variation}) on $E$ provided there is a countable collection of sets $\{E_n\}$ covering $E$ so that f has $\VB$ on $E_n$ for each $n$.
\end{definition}

\begin{definition}Let $f\colon~\mathbb R\to\mathbb R$ and let $E$ be a set of real numbers. We say that $f$ is $\AC$ ({\em absolutely continuous}) on $E$  provided that for every $\varepsilon>0$ there exists $\eta>0$ such that for every sequence of  intervals $\{[c_i,d_i]\}$ with pairwise disjoint interiors whose endpoints belong to $E$, the inequality $\sum_i d_i-c_i<\eta$ implies $\sum_i\vert f(d_i)-f(c_i)\vert<\varepsilon$. We say that $f$ is $\ACG$ ({\em generalized absolutely continuous}) on $E$ provided there is a countable collection of sets
$\{E_n\}$ covering $E$ so that f is $\AC$ on $E_n$ for each $n$.
\end{definition}

Let $f\colon~\mathbb R\to \mathbb R$ be a function that is continuous on a closed
bounded set $E$. Let $\conv(A)$ denote a convex hull of a set $A\subset \mathbb{R}$. Define a function $g\colon~\conv(E)\to\mathbb R$ 
 associated with $f$ to be  the
continuous function for which $f(x)=g(x)$ at each point $x\in E$ and that is linear on each interval contiguous to $E$ in $\conv(E)$.

\begin{remark}
Notice that if a continuous interval map has $\VB$ on a set $E$ then it also has  $\VB$ on its closure $\overline{E}$.
Consequently, if a continuous map $f$ has VBG, then the covering $\{E_n\}$ in the definition  can consist of closed sets.  
\end{remark}

\begin{lemma}\label{l:1}Let $f\colon~I\to \mathbb R$ be a continuous function and let $E\subset I$
 be such that $f'_{ap}(x)$ exists and is finite for each $x\in E$. The following is true.
\begin{enumerate}[(i)]
\item\label{l:1:i} $f$ is $\ACG$ (has $\VBG$) on the set $E$.
\item\label{l:1:ii}  $f$ satisfies Lusin's condition on the set $E$, i.e. if $E_0\subset E$ and $\lambda(E_0)=0$ then $\lambda(f(E_0))=0$.
\item\label{l:1:iii} Let $\{E_n\}$ be a countable collection of closed sets covering $E$ and such that $f$ is $\VB$ on $E_n$ for each $n$  (in particular provided by \VBG). Then
the following holds:
\begin{enumerate}[({iii}${}_1$)]
    \item\label{l:1:iii-1} For each $n$, the continuous function $g_n\colon~\conv( E_n)\to\mathbb R$ associated with $f|_{E_n}$, i.e., $f(x)=g_n(x)$ for each $x\in E_n$, is of bounded variation and
        \item\label{l:1:iii-2} 
        $f'_{ap}(x)=g'_n(x)$ for almost every $x\in E_n$.
\end{enumerate}
\end{enumerate}
\end{lemma}
\begin{proof}For the conclusion \ref{l:1:i}  see \cite[Cor. 39]{To15}. Conclusion \ref{l:1:ii} follows from the $\ACG$ property contained in \ref{l:1:i}. The property \ref{l:1:iii} has been established in \cite[Lemma 7]{To15}. In particular, the derivatives of $f$ and $g_n$ equal for $\lambda$-a.e. $x\in E_n$.
\end{proof}
A point $x_0\in (0,1)$ is a {\em point of dispersion of a measurable set $A$} if
$$
\lim_{h\to 0_{+}}\frac{\lambda(A\cap [x_0-h,x_0+h])}{2h}=0.
$$

\begin{definition}Let $f$ be measurable in a neighborhood of a point $x_0$. The {\em right
upper approximate limit of $f$ at $x_0$} is the greatest lower bound (the infimum) of the set
$$\{y: \{x>x_0\colon~f(x) > y\}\text{ has }x_0\text{ as a point of dispersion}\}$$ Now let $G(x,x_0)=\frac{f(x)-f(x_0)}{x-x_0}$, $x\neq x_0$. We define the {\em approximate
upper right derivative}, $D^+_{\ap}f(x_0)$, as the right upper approximate limit of $G(x, x_0)$ as $x\to x_0$. The other approximate extreme derivatives
$D^-_{\ap}f(x_0)$, $D_{+{\ap}}f(x_0)$, $D_{-{\ap}}f(x_0)$ are defined analogously.
\end{definition}

 A point $x$ for which $D^+_{\ap}f(x)=D^-_{\ap}f(x)=\infty$, $D_{+{\ap}}f(x)=D_{-{\ap}}f(x)=-\infty$ is called an {\em approximate knot point}.
The following version of the Denjoy-Young-Saks theorem employing approximate derivatives can be found in \cite[Theorem 7.15]{Je62}.

\begin{theorem}\label{t:1}If $f\colon~[a,b]\to \mathbb R$ is measurable, then with the possible exception of a Lebesgue measure null set, $[a,b]$ can be decomposed into two sets:
\begin{itemize}
\item[(i)] $A_1$, on which $f$ is approximately differentiable;
\item[(ii)] $A_2$ of approximate knot points.
\end{itemize}
\end{theorem}

\begin{proposition}\cite[Corollary 4.3]{AMB}\label{p:1}~Let $J\subset \mathbb R$ be an interval and let $f\colon~J\to \mathbb R$ be continuous. 
\begin{enumerate}[(i)]
\item If $D^+f(x)\ge 0$ $\lambda$-a.e.\ and $D^+f(x)>-\infty$ everywhere, then $f$ is nondecreasing on $J$. 
\item Analogously, if $D_+f(x)\le 0$ $\lambda$-a.e.\ and $D_+f(x)<\infty$ everywhere, then $f$ is nonincreasing on $J$.
\end{enumerate}
\end{proposition}

\section{Auxiliary results}

Recall that the set of all points of density of a measurable set $A\subset I$ is denoted by $D(A)$.

\begin{lemma}\label{l:2}Let $f\in C_{\lambda}$. For each Borel set $A\subset I$ 
\begin{equation*}
\lambda(\{x\in A\colon~x\in D(A)~\&~f(x)\in D(f(A))\})=\lambda(A).\end{equation*}
\end{lemma}
\begin{proof}
By regularity of the Lebesgue measure it is sufficient to verify the case when $A$ is closed.
Let us denote $N:=f(A)\setminus D(f(A))$. By Lemma \ref{l:3} and Remark~\ref{rem:fullD} the set $N$ is Borel, $\lambda(N)=0$ and, since $f\in C_{\lambda}$, also $\lambda(f^{-1}(N))=0$. At the same time,
$$A\setminus f^{-1}(N)=\{x\in A\colon~f(x)\in D(f(A))\},
$$
so since $D(A)\subset A$ and $\lambda(A)=\lambda(D(A))$, we get $$\lambda(D(A)\cap (A\setminus f^{-1}(N)))=\lambda(\{x\in A\colon~x\in D(A)~\&~f(x)\in D(f(A))\})=\lambda(A).$$
The proof is completed.
\end{proof}

For any Besicovitch function  $f\colon~I\to\mathbb R$ 
a simple argument  shows that for every subinterval $J\subset I$,
\begin{equation}\label{e:33}
\lambda(A_{1}\cap J)>0
\end{equation}
where the set $A_1$ was introduced in Theorem \ref{t:1}(i).
We do not give a formal proof since we will show much more in the following theorem.
In what follows, for any $\alpha\in\mathbb R$ and interval $J\subset I$ we will use the notation
\begin{eqnarray*}
A_{1,\alpha^+}(J)&:=&\{x\in J\colon~f'_{ap}(x)\in (\alpha,\infty)\},\\
A_{1,\alpha^-}(J)
& :=&\{x\in J \colon~f'_{ap}(x)\in (-\infty,\alpha)\}.
\end{eqnarray*}

\begin{theorem}\label{t:2}Let $f\colon~I\to\mathbb R$ be a Besicovitch function. Then for every subinterval $J\subset I$ and every $\alpha\in\mathbb R$,
\begin{equation*}
\lambda(A_{1,\alpha^+}(J))>0 \qquad \text{ and }\qquad \lambda(A_{1,\alpha^-}(J))>0.
\end{equation*}
\end{theorem}
\begin{proof} Choose an arbitrary interval $J\subset I$ and assume that for some $\alpha\in\mathbb R$, $\lambda(A_{1,\alpha^-}(J))=0$.  
By Theorem~\ref{t:1} we know that 
$\lambda((A_{1}\cup A_2)\cap J)=\lambda(J)$ and
$f'_{ap}(x)\geq \alpha$ for every $x\in A_1\cap J$.
If $\alpha\ge 0$ then the assumptions of Proposition \ref{p:1}(i) are satisfied so $f$ is nondecreasing,  contradicting the fact that $f$ is Besicovitch. If $\alpha<0$, consider the Besicovitch function $k(x)=f(x)-\alpha x$.  The sets $A_1\cap J,A_2\cap J$ are the same for both functions $f,k$, but $k'_{ap}(x)\ge 0$ for each $x\in A_1\cap J$, so by Proposition \ref{p:1}(i) $k$ is nondecreasing, again leading to a contradiction. The case when $\lambda(A_{1,\alpha^{+}}(J))=0$ can be dealt with analogously. 
\end{proof}

\section{Besicovitch functions are dense in \texorpdfstring{$C_{\lambda}$}{}}

Let us denote 
\begin{equation*}
B_{\lambda} :=\{f\in C_{\lambda}\colon~f\text{ is Besicovitch}\}.
\end{equation*}
\begin{proposition}\label{p:2}
The set $B_{\lambda}$ is dense in $C_{\lambda}$.  
\end{proposition}
\begin{proof} In \cite[Section 3]{Bo90} there had been constructed a function $f\in B_{\lambda}$ such that 
$f(0)=f(1)=0$, $f(\frac{1}{2})=1$ and the graph of $f$ is symmetric with respect to the line $x=1/2$.

Let us denote $g :=f\vert_{  [0,\frac{1}{2}]}$ and $h :=f\vert_{[\frac{1}{2},1]}$. For two nondegenerate subintervals $[a,b]$ and $[c,d]$ let us define rescaled versions of $g$ and $h$ as  
\begin{align*}&
g_{a,b;c,d}(x) :=c+(d-c)g\left (\frac{x-a}{2(b-a)}\right ),~x\in [a,b],\\    
&h_{a,b;c,d}(x) :=c+(d-c)h\left (\frac{1}{2}+\frac{x-a}{2(b-a)}\right ),~x\in [a,b].    
\end{align*}
Note that for any Borel set $A\subset [c,d]$
we have
$$
\lambda(g_{a,b;c,d}^{-1}(A))=\lambda(h_{a,b;c,d}^{-1}(A))=
\frac{b-a}{c-d}\lambda(A)=\lambda(L^{-1}(A))
$$
where $L$ is an affine function on $[a,b]$ with range $[c,d]$.

By \cite[Proposition 8]{BT} the set $PA_{\lambda}$ of piecewise affine functions is dense in $C_{\lambda}$. Fix $k\in C_{\lambda}$, $\eps>0$ and $\ell\in PA_{\lambda}$ for which $\|  k-\ell \|_\infty<\frac{\eps}{2}$. We can consider a finite set 
$$P_{\eps} :=\{0=p_0<p_1<\cdots<p_{n-1}<p_n=1\},~\max_{1\le i\le n}\vert p_i-p_{i-1}\vert<\frac{\eps}{2}$$
and $P_{\eps}$ containing all critical values of $\ell$, i.e., the values at all local extrema. 
 
We enumerate pieces of monotonicity of $\ell$ associated to $P_\eps$, so that
$$\ell^{-1}([p_{i-1},p_i])=\bigcup_j[a(i,j),b(i,j)],~i\in\{1,\dots,n\},$$
and use them to define the function $m\in C_{\lambda}$ as 
\begin{equation*}\label{e:1}
m(x) :=     \begin{cases}
              g_{a(i,j),b(i,j);p_{i-1},p_i}(x),&\text{if } x\in [a(i,j),b(i,j)],~\ell(a(i,j))<\ell(b(i,j)),\\              h_{a(i,j),b(i,j);p_{i-1},p_i}(x),&\text{if } x\in [a(i,j),b(i,j)], ~\ell(a(i,j))>\ell(b(i,j)).
           \end{cases}
\end{equation*}
It follows that $\| k-m\|_\infty <\eps$ and since $f$ was from $C_{\lambda}$, the function $m\in C_{\lambda}$.
It is clear the map $m$ is Besicovitch as well, so $m \in B_\lambda$.
\end{proof}

In  \cite{BCOT1} we showed if $f\in C(I)$ preserves a non-atomic probability measure $\mu$ with $\supp~\mu=I$ (denoted $f\in C_{\mu})$ then 
 \begin{itemize}
    \item There exists an increasing $h$ of $I$ such that $h\circ f\circ h^{-1}\in C_{\lambda}$, and
 \item $f$ has a dense set of periodic points, i.e., $\overline{\per(f)}=I$.
 \end{itemize}
So we see that the set $CP$ of all maps with dense set of periodic points can be expressed as $CP=\bigcup_{\mu}C_{\mu}$.

\begin{corollary}\label{cor:BesicovitchDenseCP}The set of Besicovitch functions is dense in $CP$ and thus dense in $\CP$.
\end{corollary}

\begin{proof} Fix a function $f\in CP$. By the above assertions there exists a measure $\mu$, a homeomorphism  $h\in C(I)$ and $g\in C_{\lambda}$ such that 
\begin{equation*}f\in C_{\mu} \quad \text{ and } \quad  g=h\circ f\circ h^{-1}\in C_{\lambda}.
\end{equation*}
Fix $\eps>0$. For a sufficiently small $\delta >0 $ choose a Besicovitch function $\tilde g \in C_\lambda$ such that 
$\| \tilde g - g\|_\infty < \delta$. 
Let  $p$ be a piecewise linear homeomorphism of $I$ such that 
$$\| h-p\|_\infty+\| h^{-1}-p^{-1}\|_\infty<\delta.$$ 
Let  $q :=p^{-1}\circ \tilde g\circ p$, clearly 
$q \in C_{\nu} \subset CP$ for some $\nu$ with $\supp\, \nu=I$. 

The maps  $\tilde g,h,h^{-1},p,p^{-1}$ and $q$ are all  uniformly 
continuous, thus we can choose $\delta>0$ sufficiently small so that
$$\|f-q\|_{\infty}<\eps.
$$
Since $p$ is piecewise linear, $q$ is a Besicovitch function.  
\end{proof}

\section{No Besicovitch function in  \texorpdfstring{$C_{\lambda}$}{} is invertible  \texorpdfstring{$\lambda$}{}-a.e.}

For the notion of invertibility $\lambda$-a.e.\ we use the following definition. It corresponds to the classical definition in ergodic theory described in \cite[Def.1.1(c) and p. 59]{Wal82}.  

\begin{definition}\label{d:1}Let $f\in C_{\lambda}$. We say that $f$ is invertible $\lambda$-a.e.\ if there is a Borel set $X\subset I$ such that $\lambda(X)=1$ and for each Borel $Y\subset I$,
\begin{equation*}
\lambda(f(Y\cap X))=\lambda(Y).
\end{equation*}
\end{definition}

\begin{theorem}\label{t:3}Let $f\in C_{\lambda}$ be a Besicovitch function. Then for every interval $I_0\subset I$ there is a Borel set $Y\subset I_0$ such that for each Borel set $X\subset I$ satisfying $\lambda(X)=1$,
\begin{equation}\label{e:11}
\lambda(f(Y\cap X))>\lambda(Y).
\end{equation}
In particular, $f$ is not invertible $\lambda$-a.e.
\end{theorem}
\begin{proof}By Theorem \ref{t:2} for any fixed $\alpha>1$ we have
$$\lambda(A_{1,\alpha^+}( I_0))>0.$$
 By Lemma \ref{l:1} the map $f$  has VBG on $ E :=A_{1,\alpha^+}( I_0)$, in particular 
there is a countable collection of sets $\{E_n\}$ covering $E$ such that $f$  has VB on each $E_n$. Since the collection is countable, there is at least one $E_n$ with $\lambda(E_n)>0$.  If the set $E_n$ had a nonempty interior, the function $f$ would have finite derivative at almost every interior point of $E_n$ which is impossible. So $E_n$ has an empty interior and w.l.o.g.\ we may assume that $\lambda(J\cap E_n)>0$ for each interval $J\subset I$ such that $\inter(J)\cap E_n\neq\emptyset$.

For each $m\in\mathbb N$ let us denote $\mathbb J_m$ the set of all closed subintervals of $I$ of the length less than or equal to $\frac{1}{m}$. Let
$$C_m :=\{x\in E_n\colon~x\in D(E_n\cap\inter(J))~\&~f(x)\in D(f(E_n\cap \inter(J))\text{ for some }J\in\mathbb J_m\}.
$$
It is a consequence of Lemma \ref{l:2} that $\lambda(C_m)=\lambda(E_n)$ for each $m$. For $m\in\mathbb N$  and $x\in E_n$ let us denote 
$$A_{m,x} :=E_n\cap \bigg[x-\frac{1}{m},x+\frac{1}{m}\bigg]$$ and put
$$B_m :=\{x\in E_n\colon~x\in D(E_n)~\&~f(x)\in D(f(A_{m,x}))\}.
$$
Clearly for each $m\in\mathbb N$, $C_m\subset B_m$ and also $\lambda(B_m)=\lambda(E_n)$. If we put $$B :=\bigcap_{m\ge 1}B_m =E_n\setminus \bigcup_{m\ge 1}(E_n\setminus B_m),$$ we get again $\lambda(B)=\lambda(E_n)>0$. 

Recall that by Lemma \ref{l:1}\ref{l:1:iii-2},  $f'_{ap}(x)=g'_n(x)$ for almost every $x\in B$. 
Fix such a point
$z\in B$ and note that by \ref{l:1:iii-1} we have:

\begin{enumerate}[(i)]
\item\label{t:3:aux-i} $f\vert A_{m,z}=g_n\vert A_{m,z}$ for each $m$,
\item\label{t:3:aux-ii} $z\in D(E_n)$ and 
$f(z)=g_n(z)\in D(g_n(A_{m,z}))$ for each $m$,
\item\label{t:3:aux-iii} $f'_{ap}(z)=g_n'(z) = \beta>\alpha>1$.
\end{enumerate}

By \ref{t:3:aux-iii} we can consider $m$ sufficiently large to ensure
\begin{align}\label{a:1}&g_n(x)> \frac{\alpha+\beta}{2}
(x-z)+g_n(z),~\text{if }x\in (z,z+\frac{1}{m}]\\
&\label{a:2}g_n(x)<
\frac{\alpha+\beta}{2}
(x-z)+g_n(z),~\text{if }x\in [z-\frac{1}{m},z).
\end{align}
By \ref{t:3:aux-ii} we have that $f(z)\in D(g_n(A_{m,z}))$, so by (\ref{a:1}),(\ref{a:2}) 
for sufficiently small $\gamma>0$ if we consider $K=[f(z)-\gamma,f(z)+\gamma]$ then 
\begin{equation}\label{e:7}
\lambda(K\cap g_n(A_{m,z}))>\frac{1}{\alpha}\lambda(K)=\frac{2\gamma}{\alpha}>2\gamma\frac{2}{\alpha+\beta}\ge\lambda(L),
\end{equation}
where $L=\{x\in [z-\frac{1}{m},z+\frac{1}{m}]\colon~g_n(x)\in K\}$, since by \ref{t:3:aux-iii} and the formulas \eqref{a:1},\eqref{a:2} and since $L\subset [z-\frac{2\gamma}{\alpha+\beta},z+\frac{2\gamma}{\alpha+\beta}]$.  Notice that the set $L$ is closed with a nonempty interior. By \ref{t:3:aux-iii} we have
\begin{equation}\label{e:8}f(L\cap A_{m,z})=g_n(L\cap A_{m,z})=K\cap g_n(A_{m,z}).
\end{equation}

Since the interior of $L\cap A_{m,z}$ is empty, we clearly have $\lambda(L)>\lambda(L\cap A_{m,z})$, so from (\ref{e:7}) and (\ref{e:8}) we obtain
\begin{equation}\label{e:9}
\lambda(f(L\cap A_{m,z}))=\lambda(K\cap g_n(A_{m,z}))>\lambda(L\cap A_{m,z}).
\end{equation}
 Put $Y=L\cap A_{m,z}\subset E_n\subset I_0$ and note that $Y$ is closed.
In order to show (\ref{e:11}),
fix any Borel set $X\subset I$ such that $\lambda(X)=1$. By Lemma \ref{l:1}(ii) the function $f$ satisfies Lusin's condition on the set $Y$, so by (\ref{e:9})
$$
\lambda(f(Y\cap X))=\lambda(f(Y))>\lambda(Y).
$$
In particular, $f$ is not $\lambda$-a.e.\ invertible due to Definition \ref{d:1}.
\end{proof}

\begin{corollary}\label{c1}If $f\in C_{\lambda}$ is a Besicovitch function then $h_{\lambda}(f)>0$.
\end{corollary}
\begin{proof}It is a known fact that if $h_{\lambda}(f)=0$ then $f$ is invertible $\lambda$-a.e.\ \cite[Cor. 4.14.3]{Wal82}.
\end{proof}

Combining Corollary \ref{c1} with Theorem 1 of \cite{BCOT} immediately yields a new proof of the following fact, see  \cite[Section 3]{Bo90}.

\begin{corollary}\label{c2}The set of all Besicovitch function in $C_{\lambda}$ is a meager set.
\end{corollary}

\section{Open problems}
In Corollary~\ref{cor:BesicovitchDenseCP} we showed that the set of Besicovitch functions is dense in $CP$.

\begin{question} Is the set of all Besicovitch functions in CP meager? 
\end{question}

It is possible that $\CP\setminus CP$ contains no Besicovitch functions. By Theorem 4.5 from \cite{BCOT2} $\CP\setminus CP$ is also dense in $\CP$. 

\begin{question}
Does there exist a Besicovitch function in $\CP\setminus CP$?
\end{question}

In \cite[Theorem 1]{CO} the authors proved that the generic map in $C_{\lambda}$
has the {\em iterated crookedness property}, i.e., for each $\delta > 0$ there exists an iterate which is $\delta$-crooked. This motivates the following question.

\begin{question}\label{q:crooked}
Does every Besicovitch $f \in C_\lambda$ have the iterated crookedness property?
\end{question}

In $C(I)$, the collection of maps with the iterated crookedness property forms a nowhere dense set \cite{BKU} and since Besicovitch functions are dense in $C(I)$\footnote{Sketch of the proof: In \cite{Pep28} the author constructed a Besicovitch function $f\in C(I)$ satisfying $f(0)=f(1)=0$ and $f\ge 0$ on $I$. 
Define two maps $f^{(i)} :I \to I$ such that 
$f^{(1)}|_{[0,1/2]}$ and $f^{(2)}|_{[1/2,1]}$ are rescalings of $f$ and $f^{(i)} \equiv 0$ on the complementary interval.
If $g$ is affine and increasing on $I$ then the function $h:= g+\alpha f^{(1)}-\beta f^{(2)}$ is in $C(I)$ for sufficiently small $\alpha, \beta \in \mathbb{R}$ and  is Besicovitch. 
Combining this with the fact that piecewise affine functions with nonzero slopes are dense in $C(I)$ yields the assertion.} there exist Besicovitch functions without iterated crookedness property, so Question \ref{q:crooked} stated for $f \in C(I)$ has a negative answer.

Maps with iterated crookedness property have topological entropy either $0$ or $\infty$ \cite{Mouron}, thus an example of a Besicovitch map with finite topological entropy would give a negative answer to Question~\ref{q:crooked}. This motivates us to ask the following question.

\begin{question} Does every Besicovitch function $f\in C_{\lambda}$ have infinite topological entropy?
\end{question}

Recent advances on the geometric analysis of nowhere differentiable functions further motivate the study of 
Besicovitch functions. In particular, the article of Bara\'nski, B\'ar\'any and Romanowska \cite{BBR} 
establishes sharp results for the Hausdorff dimension of the graph of the Weierstrass function, showing how fractal dimension reflects fine-scale oscillatory behavior beyond differentiability. It would be similarly natural to investigate the Hausdorff or lower-box dimension of the graphs of Besicovitch functions which lack both finite and infinite derivatives at every point. Determining whether such functions exhibit universal dimensional properties or a broad range of possible graph dimensions would provide an interesting geometric refinement of our knowledge about such functions.

\begin{question} Do graphs of Besicovitch functions in $C_{\lambda}$  all have the same Hausdorff (lower-box) dimension?
\end{question}

Note that generic $C_{\lambda}$ function has Hausdorff dimension and lower box dimension equal to $1$ whereas the upper box dimension equal to $2$ \cite{SW}. On the other hand, 
the measure theoretic entropy of the generic map from $C_{\lambda}$ is $0$ \cite{BCOT}.
Combined with the discussion above, we thus have the following natural question.

\begin{question} If the Hausdorff dimension of a graph of $f \in C_{\lambda}$ is greater than $1$, does this imply positive measure theoretic entropy?
\end{question}

\section*{Acknowledgments}
J. \v Cin\v c acknowledges the support of 
Slovenian research agency ARIS grant J1-4632 and ARIS project under Contract No. SN-ZRD/22-27/0552.

P. Oprocha was partially supported by project
No.~CZ.02.01.01/00/23\_021/0008759 supported by EU funds, through the Operational Programme Johannes Amos Comenius.

\begin{table}[ht]
\begin{tabular}[t]{p{2.5cm}  p{11cm} }
\includegraphics [width=2.1cm]{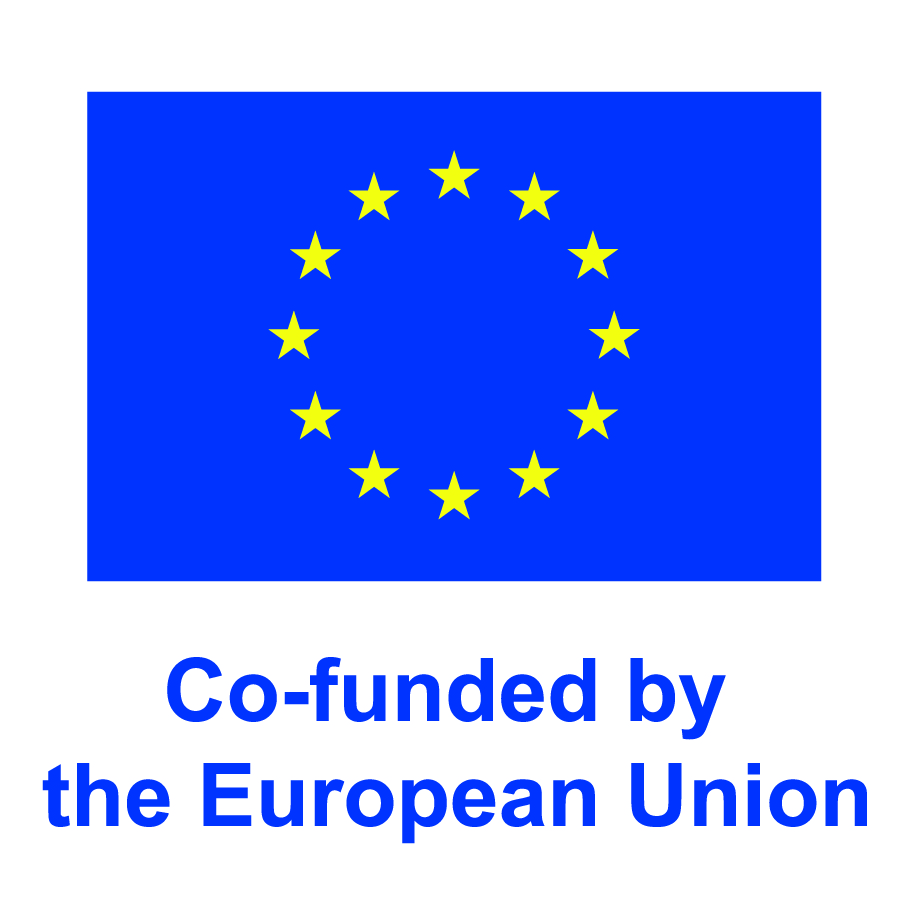} &
\vspace{-2cm}
This research is part of J.\ \v Cin\v c's  project that has received funding from the European Union's Horizon Europe research and innovation programme under the Marie Sk\l odowska-Curie grant agreement No.\ HE-MSCA-PF-PFSAIL-101063512.\\
\end{tabular}
\end{table}

\end{document}